\documentclass[12pt,a4paper]{amsart}
\usepackage{amssymb,amsthm,graphicx,enumerate}

\newtheorem{Thm}{Theorem}[section]
\newtheorem{Prop}[Thm]{Proposition}
\newtheorem{Lem}[Thm]{Lemma}
\newtheorem{Cor}[Thm]{Corollary}
\theoremstyle{definition}
\newtheorem{Def}[Thm]{Definition}
\newtheorem{Ex}[Thm]{Example}
\newtheorem{Conj}[Thm]{Conjecture}
\newtheorem{ins}{INS\'ER\'E}

\newcommand\HNN{\textnormal{HNN}}
\newcommand\Aut{\textnormal{Aut}}

\newcommand\BS{\textnormal{BS}}
\newcommand\Ker{\textnormal{Ker}}
\newcommand\Z{\mathbf{Z}}
\newcommand\C{\mathbf{C}}
\newcommand\Q{\mathbf{Q}}
\newcommand\R{\mathbf{R}}
\newcommand\K{\mathbf{K}}
\newcommand\p{\mathbf{P}}
\newcommand\F{\mathbf{F}}
\newcommand\apx{\stackrel{\centerdot}\approx}

\newcommand{\Conv}{\mathop{\scalebox{2.0}{\raisebox{-0.25ex}{$\ast$}}}}%

\newcommand\bbi{\begin{ins}\marginpar{!!!!!!!!!!!!!!!!!!!}}
\newcommand\eei{\end{ins}}

\begin{document}

\title[Embeddings of generalized Baumslag-Solitar groups]{On equivariant embeddings of generalized Baumslag-Solitar groups}

\author[Y.~Cornulier]{Yves Cornulier}
\address[Y.C.]{Laboratoire de Math\'ematiques\\
B\^atiment 425, Universit\'e Paris-Sud 11\\
91405 Orsay\\FRANCE}
\email{yves.cornulier@math.u-psud.fr}

\author[A.~Valette]{Alain Valette}
\address[A.V.]{Institut de Math\'ematiques\\
Rue Emile Argand 11\\
Unimail\\
CH-2000 Neuch\^atel\\
SWITZERLAND}
\email{alain.valette@unine.ch}
\date{December 30, 2012}

\subjclass[2010]{Primary 20E08; Secondary 20F69, 30L05, 43A07, 46B85}%


\baselineskip=16pt

\maketitle

\begin{abstract} Let $G$ be a group acting cocompactly without inversion on a tree $X$, with all vertex and edge stabilizers isomorphic to the same free abelian group $\Z^n$. We prove that $G$ has the Haagerup Property if and only if $G$ is weakly amenable, and we give a necessary and sufficient condition for this to happen. In particular, denoting by $d$ the rank of the fundamental group of the graph $G\backslash X$, we deduce that $G$ has the Haagerup Property if either $d=0$, or $d=1$, or $n=1$. In these three cases, we show that the $L^p$-compression rate of $G$ is 1, and that its equivariant $L^p$-compression rate is $\max\{1/p,1/2\}$ (provided $G$ is non-amenable). We also discuss quasi-isometric embeddings of $G$ into a product of finitely many regular trivalent trees.
\end{abstract}

\section{Introduction}

The class of groups studied in this paper is the class $\mathcal{GBS}$ of Generalized Baumslag-Solitar groups, i.e.\ the groups acting cocompactly without inversion on a tree, with all vertex and edge stabilizers being free abelian groups of the same finite rank. If this rank is $n$, we denote by $\mathcal{GBS}_n$ the corresponding subclass. These groups have been considered e.g.\ in \cite{Beeker,Whyte2}. The ordinary Baumslag-Solitar groups are obtained for $n=1$ and actions on trees that are transitive on vertices and edges.

This paper consists of two rather distinct parts, however linked by the employed techniques. In the first part, we deal with weak forms of amenability, namely {\it Haagerup Property} (a.k.a.\ {\it a-T-menability}) and {\it weak amenability}.

\begin{Def} A locally compact group $G$ has the {\it Haagerup Property} if there exists a net of continuous $C_0$ positive definite functions on $G$, and converging to 1 uniformly on compact sets.
\end{Def}

\begin{Def}\label{wa} A locally compact group $G$ is {\it weakly amenable} if there exists a net $(\varphi_i)_{i\in I}$ of continuous functions with compact support on $G$, converging to 1 uniformly on compact sets, such that $\sup_{i\in I}\|\varphi_i\|_{\textnormal{cb}}<+\infty$.
\end{Def}

Here $\|.\|_{\textnormal{cb}}$ denotes the completely bounded norm on multipliers of the Fourier algebra of $G$. If $G$ is weakly amenable, the {\it Cowling-Haagerup constant} $\Lambda(G)$ is the best possible constant $\Lambda$ such that $\sup_{i\in I}\|\varphi_i\|_{\textnormal{cb}}\leq\Lambda$, where $(\varphi_i)_{i\in I}$ is a net as in Definition \ref{wa}.

We refer to \cite{CCJJV} (resp.\ \cite{CH}) for background on the Haagerup Property (resp.\ weak amenability).

Motivated by the study of the class $\mathcal{GBS}$, we consider in section 2 the situation of a discrete group $\Gamma$ endowed with a linear representation $\rho$ on $\R^n$, and study the question of the Haagerup Property for the semidirect product $\R^n\rtimes_\rho\Gamma$; we prove a result of independent interest:

\begin{Prop} Assume that $\Gamma$ has the Haagerup Property (respectively is weakly amenable). The following are equivalent:
\begin{enumerate}[(i)]
  \item $\R^n\rtimes_\rho\Gamma$ has the Haagerup Property (resp.\ is weakly amenable);
  \item $\R^n$ carries a $(\R^n\rtimes_\rho\Gamma)$-invariant mean on Borel subsets;
  \item the closure $\overline{\rho(\Gamma)}$ in $\textnormal{GL}_n(\R)$, is amenable.
\end{enumerate}
\end{Prop}
Note that the $(\R^n\rtimes_\rho\Gamma)$-invariance means that the mean is both invariant by the action of $\Gamma$ and by the action of translations.

We now describe our main results. Given $G\in \mathcal{GBS}_n$, Bass-Serre theory \cite{Se} allows to view $G$ as the Bass-Serre fundamental group of a graph of groups $(X,\Z^n)$, where $X$ is a finite graph. The choice of a maximal tree $T$ in $X$ provides a presentation of $G$ including $d$ stable letters, where $d$ is the number of edges of $X$ not in $T$ (so $d$ is also the rank of the usual fundamental group of $X$). These $d$ stable letters generate inside $G$ a copy $H_d$ of the free group of rank $d$, so that $G$ appears as a semidirect product $G=N\rtimes H_d$, where $N$ is the normal subgroup generated by vertex groups. Our first task will be to construct a homomorphism $\mu$ from $G$ to the affine group $\R^n\rtimes \textnormal{GL}_n(\R)$, such that $\mu(N)\subset \R^n$ (actually, for every vertex group $G_v$, the image $\mu(G_v)$ is a lattice in $\R^n$), while $\mu(H_d)\subset \textnormal{GL}_n(\R)$. The construction is the natural generalization of Levitt's modular homomorphism for $n=1$ \cite[Section 2]{Lev}.

\begin{Ex}\label{goodexamples}
\item[1)] Consider the Baumslag-Solitar group $\BS(m,n)$ in its standard presentation:
$$\BS(m,n)=\langle t,b|tb^mt^{-1}=b^n\rangle.$$
The homomorphism $\mu: \BS(m,n)\rightarrow \R\rtimes \R^\times$ consists of mapping $b$ to the translation by 1 and $t$ to the dilation by $\frac{n}{m}$; it plays a crucial role in the Gal-Januszkiewicz paper \cite{GaJa}, who used it to prove that $\BS(m,n)$ has the Haagerup Property (the same proof also showing weak amenability with constant 1, even if this not explicitly mentioned).
\item[2)] Realize the free group $\F_2$ on 2 generators in $\textnormal{SL}_2(\Z)$, as the subgroup generated by the matrices $\left(\begin{array}{cc}1 & 2 \\0 & 1\end{array}\right)$ and $\;\left(\begin{array}{cc}1 & 0 \\2 & 1\end{array}\right)$. Consider the semidirect product $\Z^2\rtimes\F_2$, and let it act on the Cayley tree of $\F_2$ by means of the quotient map: this action has all stabilizers isomorphic to $\Z^2$, however, as it is well-known, $\Z^2\rtimes\F_2$ neither has the Haagerup Property, nor is weakly amenable; this shows that Haagerup Property or weak amenability are not preserved under graphs of groups. Here $\mu$ is just the inclusion of $\Z^2\rtimes \F_2$ into $\R^2\rtimes \textnormal{GL}_2(\R)$.
\end{Ex}

Coming back to the general case: let $\overline{\iota(G)}$ be the closure of the image of $G$ in the automorphism group of the Bass-Serre tree associated with $(X,\Z^n)$. Generalizing a result of Gal and Januzkiewicz \cite{GaJa} for $\BS(m,n)$, we first realize $G$ as a discrete subgroup in a product of a totally disconnected group and a (not connected) Lie group:

\begin{Prop}\label{proper} The inclusion $G\rightarrow\overline{\iota(G)}\times(\R^n\rtimes \mu(H_d)):g\mapsto(g,\mu(g))$ is proper.
\end{Prop} 

From this we deduce our first main result, namely that the closure $\overline{\mu(H_d)}$ (in the standard topology of $\textnormal{GL}_n(\R)$) controls both Haagerup Property and weak amenability with constant 1 for $G$.

\begin{Thm}\label{Mainthm} For $G\in \mathcal{GBS}$, the following are equivalent:
\begin{enumerate}
\item[i)] $\overline{\mu(H_d)}$ is amenable;
\item[ii)] $G$ has the Haagerup Property;
\item[iii)] $G$ is weakly amenable;
\item[iv)] $\Lambda(G)=1$.
\end{enumerate}
\end{Thm}

For general finitely generated groups, it was proved by the authors along with Y.~Stalder \cite{CSV} that the Haagerup Property does not imply weak amenability; the converse is an open question.

There are three cases in Theorem \ref{Mainthm} where $\overline{\mu(H_d)}$ is clearly amenable, yielding the following corollary:

\begin{Cor}\label{suffHaagerup} In each of the following cases, $G$ has the Haagerup Property and satisfies $\Lambda(G)=1$:
\begin{enumerate}
\item[1)] $d=0$, i.e.\ $X$ is a tree;
\item[2)] $d=1$, i.e.\ $X$ has the homotopy type of a circle;
\item[3)] $n=1$, i.e.\ $G\in \mathcal{GBS}_1$.
\end{enumerate}
\end{Cor}

For the usual Baumslag-Solitar groups, this appears in \cite{GaJa}; an observation of Kropholler \cite{Kro} implies that groups in $\mathcal{GBS}_1$ are free-by-\-solvable (actually their second derived subgroup is free), which implies the Haagerup Property; a similar argument can been carried out if $\mu(H_d)$ is solvable, in which case some iterated derived subgroup of $G$ is free (see lemma \ref{Kermufree} below). If $\mu(H_d)$ is not virtually solvable, then we show that $G$ is not free-by-amenable (Proposition \ref{virtsol}), and thus this alternative approach falls apart.

Haagerup Property for amalgamated products of $\Z^n$'s was treated in \cite{CCJJV}. All the other examples provided by Corollary \ref{suffHaagerup}, appear to be new.

\bigskip
Let us move now to the second part of the paper. Let $(X,d)$ be a metric space. In the large scale study of $X$, it is natural to focus on the behavior of coarse embeddings of $X$ into Hilbert spaces or more general Banach spaces. Recall that a coarse map of $X$ into a Banach space $E$ is a mapping $X\to E$ such that 
$$\sup\{\|f(x)-f(y)\|:\;x,y\in X,\;d(x,y)\le r\}<\infty,\quad\forall r>0.$$
If $f:X\to E$ is a coarse map, its {\em compression function} is defined by
$$\rho_f(r)=\inf\{\|f(x)-f(y)\|:\;x,y\in X,\; d(x,y)\ge r\}$$
(where $\inf\emptyset=+\infty$). If $\rho_f$ is a proper map (i.e.\ $\lim_\infty\rho_f=\infty$), $f$ is called a {\it coarse embedding} (although {\it uniformly proper} would have been more accurate). If $\rho_f$ grows slowly to infinity, it means that the embedding $f$ has a lot of distortion. At the opposite, if $f$ is a quasi-isometric embedding, then $\rho_f$ grows at least linearly. If $X$ is geodesic and unbounded, this is the largest possible behaviour for $\rho_f$. When $E$ is fixed, the class $\mathcal{E}(X,E)$ of possible asymptotic behaviors for $\rho_f$, when $f$ ranges over coarse embeddings $X\to E$, is a quasi-isometry invariant of $X$. Let $\mathcal{T}$ be the 3-regular tree (with the graph distance). In case $E$ is a Hilbert space or more generally an $L^p$-space, $\mathcal{E}(\mathcal{T},E)$ is well understood, by results of Tessera \cite{Tessera}.

A fashionable (but much less fine) numerical invariant associated with $\mathcal{E}(X,E)$ (when $X$ is any metric space), is the {\it metric $E$-compression}, i.e.\ the supremum $\alpha_E(X)$ of $\lambda$ such that $\mathcal{E}(X,E)$ contains the asymptotic behavior of the function $x\mapsto x^\lambda$. Write $\alpha_p(X)=\alpha_{L^p}(X)$; it is known that $\alpha_p(\mathcal{T})=1$ for all $p$.

Our idea is to try to exploit the embedding in Proposition \ref{proper} to deduce metric information about groups in $\mathcal{GBS}$. We can prove that the map in Proposition \ref{proper} is a quasi-isometric embedding. To go further, we need to understand $\R^n\rtimes\mu(H_d)$ from a metric point of view. This leads to a question on which we only have partial results (see \S\ref{posconj}).

\begin{Conj}\label{quest} Let $\Gamma$ be a finitely generated group, and $\rho$ a linear representation of $\Gamma$ on $\R^n$. Then the canonical homomorphism $\R^n\rtimes_\rho\Gamma \rightarrow\R^n\rtimes\overline{\rho(\Gamma)}$ is a quasi-isometry in restriction to $\R^n$ (endowed in both terms with the restriction of the word length with respect to a compact generating subset).
\end{Conj}

A positive answer would say that $\R^n$ is not be more distorted in $\R^n\rtimes\overline{\rho(\Gamma)}$ than it is in $\R^n\rtimes_\rho\Gamma$. We provide a sufficient condition for a positive answer, which covers all cases of Corollary \ref{suffHaagerup} (i.e.\ $d=0$, or $d=1$, or $n=1$).

\begin{Thm}\label{comp1} Assume that there exists a closed, almost connected subgroup $L$ of $\textnormal{GL}_n(\R)$, such that $L$ contains $\overline{\mu(H_d)}$ as a co-compact subgroup; and also that Conjecture \ref{quest} has a positive answer for the pair $(H_d,\mu)$. Then $\alpha_p(G)=1$ for $1\leq p<\infty$.
\end{Thm}

The mapping in Proposition \ref{proper} being a homomorphism, it can be hoped to also provide information about equivariant embeddings. If some group $G$ acts by isometries on $X$, define the {\it equivariant $E$-compression} $\alpha^G_E(X)$ in the same way as $\alpha_E(X)$, but only considering the class of {\em equivariant} embeddings $X\to E$ (i.e.\ associated to an action of $G$ by affine isometries). If $\Gamma$ is a finitely generated group acting on itself by left translations, define $\alpha_p^\sharp(\Gamma)=\alpha^\Gamma_p(\Gamma)$. Again, the next result covers all cases of Corollary \ref{suffHaagerup} (i.e.\ $d=0$, or $d=1$, or $n=1$).

\begin{Thm} Under the same assumptions as Theorem \ref{comp1}, plus the fact that $\overline{\mu(H_d)}$ is amenable, we have, for $1\leq p<\infty$:
$$\alpha_p^\sharp(G)=\left\{
\begin{array}{cc}1 & \mbox{if $G$ is amenable;} \\
\max\{1/p,1/2\} & \mbox{if $G$ is not amenable.}\end{array}\right.$$
\end{Thm}

Note that in case $E$ is isomorphic, as a Banach space, to $E\times E$ (as we henceforth assume), then it does not make any difference between $X$ and $X^k$ for any integer $k$. This motivates the following definition: we say that $X\apx Y$ if there exists, for some $k,\ell$, quasi-isometric embeddings $X\to Y^k$ and $Y\to X^\ell$. It follows then that $\mathcal{E}(X,E)=\mathcal{E}(Y,E)$.

The class of spaces $\apx$-equivalent to $\mathcal{T}$ is interesting to study. It contains non-elementary word-hyperbolic groups \cite{BDS}, some solvable groups like $\textnormal{SOL}$ or the standard lamplighter group. Trivial examples outside the class of groups quasi-isometrically embedding into a finite product of trees of bounded degree, are groups with infinite asymptotic dimension such as the wreath product $\Z\wr\Z$. A less trivial example outside this class is the Heisenberg group, see Example \ref{Heis} below.

We add more examples to the class of groups in the $\apx$-equivalence class of $\mathcal{T}$.

\begin{Thm}\label{classofT} If either $d=0$, or $d=1$ and $\mu$ has no eigenvalue of modulus 1, or $n=1$, then $G$ embeds quasi-isometrically in a product of finitely many 3-regular trees. If moreover $G$ is non-amenable, then $G\apx\mathcal{T}$.
\end{Thm}

Note that, in case $d=1$, in general $G$ does not embed in a product of trees, as the next example shows.

\begin{Ex}\label{Heis} Let $\Gamma$ be the discrete Heisenberg group, viewed as the Bass-Serre fundamental group of one loop of $\Z^2$'s, with one inclusion given by the identity and the other by the matrix $\mu=
\left(
\begin{tabular}{cc}
1 & 1 \\
0 & 1
\end{tabular}\right)$. By a result of Pauls \cite{Pauls}, $\Gamma$ does not embed quasi-isometrically into a $\textnormal{CAT}(0)$-space, in particular into any product of trees. 
\end{Ex}

{\bf Acknowledgements:} We thank D.~Dreesen, A.~Iozzi, and G.~Levitt for interesting discussions.

 \setcounter{tocdepth}{1}
\tableofcontents

\section{Semi-direct products $\R^n\rtimes \Gamma$}

\subsection{Relative Property (T)}

\begin{Def} Let $X$ be a closed subset of the locally compact group $G$. The pair $(G,X)$ has the {\it relative Property (T)} if, for every net of positive definite functions on $G$ converging to 1 uniformly on compact subsets, the convergence is uniform on $X$.
\end{Def}

It is clear that, if $X$ is compact, the pair $(G,X)$ has the relative Property (T); if $G$ has the Haagerup Property and $(G,X)$ has the relative Property (T), then $X$ must be compact.  We refer to \cite{Cor} for background on the relative Property (T).

In all this section, $\K$ is a local field, $\Gamma$ is a discrete group, and $\rho: \Gamma\rightarrow \textnormal{GL}_n(\K)$ is a homomorphism. Let $\rho^*: \Gamma\rightarrow \textnormal{GL}(\K^{n*})$ be the contragredient representation.  Our aim is to prove:

\begin{Prop}\label{relT} The following are equivalent:
\begin{enumerate}
\item[i)] The closure $\overline{\rho(\Gamma)}$ is non-amenable.
\item[i')] The closure $\overline{\rho^*(\Gamma)}$ is non-amenable.
\item[ii)] There exists a non-zero $\Gamma$-invariant subspace $V\subset\K^n$ such that the projective space $\p(V^*)$ of the dual of $V$, carries no $\Gamma$-invariant probability measure. 
\item[ii')] There exists a non-zero $\Gamma$-invariant subspace $V\subset\K^n$ such that the pair $(V\rtimes\Gamma,V)$ has relative Property (T).
\item[iii)] There exists no $(\K^n\rtimes\Gamma)$-invariant mean on Borel subsets of $\K^n$.
\end{enumerate}
If $\Gamma$ has the Haagerup Property, this is still equivalent to:
\begin{enumerate}
\item[iv)] The semidirect product $\K^n\rtimes\Gamma$ does not have the Haagerup Property.
\end{enumerate}
If $\Gamma$ is weakly amenable, this is still equivalent to:
\begin{enumerate}
\item[v)] The semidirect product $\K^n\rtimes\Gamma$ is not weakly amenable.
\end{enumerate}
\end{Prop}

We will need a lemma:

\begin{Lem}\label{suffrel(T)} Assume that $\overline{\rho(\Gamma)}$ is non-amenable, and that $\overline{\rho(\Gamma)|_W}$ is amenable for every proper $\Gamma$-invariant subspace $W$ of $\K^n$. Then there exists no $\Gamma$-invariant probability measure on the projective space $\p(\K^{n*})$ associated with the dual of $\K^n$.
\end{Lem}

\begin{proof}
Assume by contradiction that there exists a $\Gamma$-invariant probability measure $\mu$ on $\p(\K^{n*})$. As $\overline{\rho(\Gamma)}$ is non-compact, by a famous result of Furstenberg (see e.g.\ \cite{Zim}, Cor. 3.2.2), $\mu$ is supported on a canonical finite set $W_1,...,W_k$ of proper projective subspaces of $\p(\K^{n*})$; in particular $\Gamma$ permutes the $W_i$. Let us assume $W_1$ to be of minimal dimension. Let $\Lambda$ be a finite index subgroup in $\Gamma$ such that $W_1$ is $\Lambda$-invariant. Let $$V_1=\{v\in\K^n: \;\langle v,w\rangle=0\;\mbox{for every}\,w\in W_1\}$$ be the orthogonal of $W_1$ in $\K^n$, so that $V_1$ is $\rho(\Lambda)$-invariant.

View $\nu:=\frac{\mu|_{W_1}}{\mu(W_1)}$ as a $\Lambda$-invariant probability measure on $W_1$; since $\nu$ is not supported on a finite union of projective subspaces of $W_1$, by Furstenberg's result again, the closure of the image of $\Lambda$ in $\textnormal{Aut}(W_1)$ is compact, hence the closure of the image of $\Lambda$ in $\textnormal{GL}(V/V_1)$ is amenable. By assumption, the closure of the image of $\Lambda$ in $\textnormal{GL}(V_1)$ is amenable; using the fact that the kernel of the map $\textnormal{Stab}_{\textnormal{GL}_n(\K)}(V_1)\rightarrow \textnormal{GL}(V_1)\times \textnormal{GL}(V/V_1)$ is abelian, we deduce that $\overline{\rho(\Lambda)}$ is amenable, which is a contradiction.
\end{proof}

We now prove Proposition \ref{relT}.

\begin{proof} $(i)\Leftrightarrow(i')$ follows from the canonical isomorphism $\textnormal{GL}(K^n)\rightarrow \textnormal{GL}(K^{n*}):A\mapsto (^tA)^{-1}$.

$(i)\Rightarrow(ii)$ follows from Lemma \ref{suffrel(T)}, by taking for $V$ a $\Gamma$-invariant subspace such that $\overline{\rho(\Gamma)|_V}$ is non-amenable, and minimal for that property.

$(ii)\Leftrightarrow(ii')$ is Proposition 3.1.9 in \cite{Cor}.

$(ii')\Rightarrow (i)$ By contraposition, assume that $\overline{\rho(\Gamma)}$ is amenable; then so is $H:=\K^n\rtimes\overline{\rho(\Gamma)}$. Let $V$ be a non-zero $\Gamma$-invariant subspace of $\K^n$; denote by $\pi$ the representation of $V\rtimes\Gamma$ obtained by composing the inclusion $V\rtimes\Gamma\hookrightarrow \K^n\rtimes\overline{\rho(\Gamma)}=H$ with the left regular representation of $H$ on $L^2(H)$. Then $\pi$ weakly contains the trivial representation, but has no non-zero $V$-invariant vector, as $V$ is non-compact in $H$. So the pair $(V\rtimes\Gamma,V)$ does not have relative Property (T).

$(i')\Rightarrow(iii)$ Assume by contradiction that $\overline{\rho^*(\Gamma)}$ is non-amenable, and there exists a $(\K^n\rtimes\Gamma)$-invariant mean on the Borel subsets of $\K^n$. From the latter assumption, for every proper $\Gamma$-invariant subspace $V\subset\K^n$, there exists a $((\K^n/V)\rtimes\Gamma)$-invariant mean on the Borel subsets of $\K^n/V$, hence a $\Gamma$-invariant mean on Borel subsets of $(\K^n/V)\backslash\{0\}$. Pushing forward to the projective space $\p(\K^n/V)$ (and appealing to the Riesz representation Theorem), there exists a $\Gamma$-invariant probability measure on Borel subsets of $\p(\K^n/V)$. This contradicts lemma \ref{suffrel(T)} when applied to a $\Gamma$-invariant subspace $W\subset(\K^n)^*$ minimal for the property that  $\overline{\rho^*(\Gamma)|_W}$ is non-amenable, since $W^*=\K^n/W^\perp$, where $W^\perp$ is the orthogonal of $W$.

$(iii)\Rightarrow (i)$ By contraposition, assuming $\overline{\rho(\Gamma)}$ to be amenable, we let it act on the compact convex set of $\K^n$-invariant means on $\K^n$; any $\overline{\rho(\Gamma)}$-fixed point will be a $(\K^n\rtimes\Gamma)$-invariant mean.

Observe that $(ii')\Rightarrow(iv)$ is always true: if the pair $(V\rtimes\Gamma,V)$ has relative Property (T), then so does the pair $(\K^n\rtimes\Gamma,V)$, preventing $\K^n\rtimes\Gamma$ from being Haagerup. Similarly, $(iii)\Rightarrow(v)$ is always true, as a particular case of a result of Ozawa (Theorem A in \cite{Oz}).

Assume now that $\Gamma$ is Haagerup (resp. weakly amenable). To prove $(iv)\Rightarrow (i)$ (resp. $(v)\Rightarrow(i)$), assume by contraposition that $\overline{\rho(\Gamma)}$ is amenable (so that $H$ is also amenable) and consider the map:
$$\K^n\rtimes\Gamma\rightarrow \Gamma\times H:(x,\gamma)\mapsto(\gamma,(x,\rho(\gamma));$$
it identifies $\K^n\rtimes\Gamma$ with a closed subgroup of $\Gamma\times H$. Since by assumption $\Gamma\times H$ is Haagerup (resp. is weakly amenable), so is $\K^n\rtimes\Gamma$.
\end{proof}

\subsection{Positive results on Conjecture \ref{quest}}\label{posconj}

Here are our partial results regarding Conjecture \ref{quest} from the introduction. Let $\Gamma$ be a finitely generated group acting on $\R^n$ by a linear representation $\rho$. Recall that $\rho$ is {\em distal} if all eigenvalues of all elements have modulus 1, and that $v\in\R^n$ is {\em exponentially distorted} in $\R^n\rtimes\Gamma$ if $\overline{\lim}_{m\rightarrow\infty}\frac{|mv|_S}{\log(m)}<\infty$, where $|\cdot|_S$ denotes the word length in $\R^n\rtimes\Gamma$. We say that $\R^n$ is {\em exponentially distorted} if every $v\in\R^n$ is.

The set of elements of $\R^n$ which are exponentially distorted in $\R^n\rtimes\Gamma$, is a $\rho(\Gamma)$-invariant subspace $\textnormal{Exp}_\Gamma(\R^n)$. 

\begin{Lem}\label{distalquotient} The action of $\Gamma$ on $\R^n/\textnormal{Exp}_\Gamma(\R^n)$ is distal.
\end{Lem}

\begin{proof} If not, some element of $\Gamma$ acts with an eigenvalue of modulus $\neq 1$; this element acts irreducibly on a line or plane, and a simple verification then shows that elements of this line or plane are exponentially distorted in $(\R^n/\textnormal{Exp}_\Gamma(\R^n))\rtimes \Gamma$; it immediately follows that lifts of these elements are exponentially distorted in $\R^n\rtimes \Gamma$, contradicting the definition of $\textnormal{Exp}_\Gamma(\R^n)$. 
\end{proof}

In the following proposition, part (ii) can be considered as the ``generic" setting of Conjecture \ref{quest}.

\begin{Prop}\label{biLip2} \begin{enumerate}
\item[a)] Conjecture \ref{quest} holds under the assumption that there exists a finitely generated subgroup $\Lambda\subset \Gamma$ such that $\rho(\Lambda)$ is a cocompact lattice in $\overline{\rho(\Gamma)}$. 
\item[b)] The following are equivalent
\begin{itemize}
\item[i)] $\R^n$ is exponentially distorted in $\R^n\rtimes\Gamma$;
\item[ii)] $\R^n$ is exponentially distorted in $\R^n\rtimes\overline{\rho(\Gamma)}$; 
\item[iii)] There is no invariant subspace of the dual $\R^{n*}$ on which the action is distal.
\end{itemize}
\end{enumerate}
\end{Prop}

\begin{proof} (a) We have to check that in some cases, in restriction to $\R^n$, the inclusion $\R^n\rtimes_\rho H_1\rightarrow \R^n\rtimes H_2$ is quasi-isometric. Our goal is to prove it for $H_1=\Gamma$ and $H_2=\overline{\rho(\Gamma)}$. First observe that it holds when $H_1\in\{\Lambda,\Gamma\}$ and $H_2=H_1/\Ker(\rho)$. So we can assume that $\rho$ is injective. Thus $\R^n\rtimes\Lambda$ stands as a closed cocompact subgroup of $\R^n\rtimes\rho(\Gamma)$, so the composite map $\R^n\rtimes\Lambda\to\R^n\rtimes\Gamma\to\R^n\rtimes\overline{\rho(\Gamma)}$ is quasi-isometric in restriction to $\R^n$. Since both maps are large-scale Lipschitz in restriction to $\R^n$, it follows that they are quasi-isometries in restriction to $\R^n$.

(b) $(i)\Rightarrow(ii)$ is clear. 

$(ii)\Rightarrow(iii)$ If $\R^n$ is exponentially distorted and there is some nonzero distal subspace in the dual, by passing to the quotient we can suppose the action to be distal and $n\ge 1$. By \cite{CG}, passing again to the quotient, we can suppose that the action on $\R^n$ preserves a scalar product. This is in contradiction with exponential distortion. 

$(iii)\Rightarrow(i)$ Suppose that there is no nonzero invariant distal subspace in the dual.  Then there is no nonzero distal quotient $\R^n/W$, with $W$ a proper $\Gamma$-invariant subspace. The result then follows from lemma \ref{distalquotient}.
\end{proof} 

\section{Some Bass-Serre theory}

We recall some relevant definitions.

A {\it graph} is a pair $X=(V,E)$ where $V$ is the set of vertices, $E$ is the set of oriented edges; $E$ is equipped with a fixed-point free
involution $e\mapsto\overline{e}$ and with maps $E\rightarrow V: e\mapsto e_+$ and $E\rightarrow V: e\mapsto e_-$ ($e_-$ is the origin and $e_+$ the
extremity of the edge $e$), such that $\overline{e}_+=e_-$ for every $e\in E$. An {\it orientation} $A$ of $X$ is the choice of a fundamental
domain for the involution on $E$.

A {\it graph of groups} $({\mathcal G}, X)$ is the data of a connected graph $X=(V,E)$ and, for every $v\in V$ a group $G_v$, for every edge
$e\in E$ a group $G_e$ such that $G_e =G_{\overline{e}}$ and a monomorphism $\sigma_e: G_e\rightarrow G_{e_+}$.

Let $F(E)=\left\langle (t_e)_{e\in E}\mid\emptyset\right\rangle$ be the free group on $E$. Denote by $F(\mathcal{G},X)$ the quotient of the free product $(\Conv_{v\in V}G_v)\ast F(E)$ by the following 
set of relations:
\begin{equation}
\left\{
\begin{array}{ll}
t_e\sigma_e(g_e)t_e^{-1}=\sigma_{\overline{e}}(g_e) & (e\in E,\,g_e\in G_e)\\
t_et_{\overline{e}}=1 & (e\in E) \\
 \end{array}\right.
\end{equation}
  
Let $T$ be a maximal tree in $X$, and $v_0\in V$ be a base-vertex.  The {\it Bass-Serre fundamental group} $G=\pi_1({\mathcal G},X,T)$ is the quotient of $F(\mathcal{G},X)$ by the relations 
$$t_e=1\;(e\in E(T))$$
 where $E(T)$ is the edge set of $T$. 

 Alternatively, the Bass-Serre fundamental group $\pi_1({\mathcal G},X,v_0)$ can be described as the subgroup of $F(\mathcal{G},X)$ consisting in elements $g_{v_0}t_{e_1}g_{v_1}t_{e_2}\penalty0 \dots\penalty0 t_{e_n}g_{v_n}$ where $e_1,\dots,e_n$ is a circuit of length $n$ in $X$, starting and ending at $v_0$, where $v_i=(e_i)_+=(e_{i+1})_-$ (for $1\leq i\leq n$), and $g_i\in G_{v_i}$ for every $i$. It is proved in \cite{Se}, Proposition 20 in Chapter 1, that the quotient map $F(\mathcal{G},X)\rightarrow \pi_1({\mathcal G},X,T)$ induces an isomorphism  $\pi_1({\mathcal G},X,v_0)\rightarrow\pi_1({\mathcal G},X,T)$.

The Bass-Serre fundamental group $G$ of a graph of groups $(\mathcal{G},X)$ acts on a tree $\tilde{X}=(\tilde{V},\tilde{E})$, where $\tilde{V}=\coprod_{v\in V}G/G_v$ and $\tilde{E}=\coprod_{e\in E}G/G_e$. Denote by $\iota:G\to\textnormal{Aut}(\tilde{X})$ the corresponding homomorphism. Define $(gG_e)_+=gG_{e_+}$ and $(gG_e)_-=gG_{e_-}$ for $e\in T$, and $(gG_e)_+=gG_{e_+}$ and $(gG_e)_-=gt_e^{-1}G_{e_-}$ for $e\in E\smallsetminus T$. The tree $\tilde{X}$ is the {\it universal cover} of $(\mathcal{G},X)$; for all this, see Theorem 12 in Chapter I of \cite{Se}. It is clear that $\tilde{X}$ is locally finite if and only if, for every edge $e$, the image of every injection $\sigma_e$ has finite index in $G_{e_+}$. 

Choosing both a maximal tree $T$ and an orientation $A$ in $X$, we denote by $H_{T,A}$ the subgroup of $G$ generated by the stable letters $t_e$'s, with $e\in A\smallsetminus T$. It is known that $H_{T,A}$ is the free group on $A\smallsetminus T$, which we can identify with the fundamental group (in the usual sense) of the graph $X$. There is then a natural epimorphism 
$$\varphi:G=\pi_1({\mathcal G},X,T) \twoheadrightarrow H_{T,A}: \left\{
\begin{array}{ccc}
g_v\mapsto 1 & for & g_v\in G_v\\
t_e\mapsto t_e & for & e\in A\smallsetminus T
\end{array}\right..$$

\begin{Lem}\label{phiextends} Assume that $\tilde{X}$ is locally finite, so that $\textnormal{Aut}(\tilde{X})$ is a locally compact group (with the topology of pointwise convergence). The homomorphism $\varphi:G=\pi_1({\mathcal G},X,T) \twoheadrightarrow H_{T,A}$ factors through a continuous homomorphism $\bar{\varphi}$ on the closure $\overline{\iota(G)}$ of $G$ in $\textnormal{Aut}(\tilde{X})$.
\end{Lem}

\begin{proof} Define a map $c:\tilde{V}\rightarrow H_{T,A}: gG_v \mapsto \varphi(g)$ (this is well-defined as $G_v\subset \Ker\varphi$). Then $c(gx)=\varphi(g)c(x)$ for $x\in\tilde{V},g\in G$. Rewriting $\varphi(g)=c(gx)c(x)^{-1}$ allows to extend $\varphi$ from $G$ to $\overline{\iota(G)}$.
\end{proof}

\section{Graphs of $\Z^n$'s}

Let $({\mathcal G},X)$ be a graph of groups. We assume that $X$ is finite, so that it has the homotopy type of a bouquet of $d$ circles, with $d=|A\smallsetminus T|$. We write $H_d$ rather than $H_{T,A}$. 

We also assume that all vertex groups and edge groups are equal to $\Z^n$. So every injection $\sigma_e$ is given by some $(n\times n)$-matrix with integer coefficients and non-zero determinant, which we also denote by $\sigma_e$. Let $G$ be the Bass-Serre fundamental group of this graph of groups. 

\subsection{Amenability of $G$}

In this subsection, we characterize the few exceptional cases where $G$ is amenable.

\begin{Prop}\label{amen} Let $G$ be the Bass-Serre fundamental group of a finite graph of $\Z^n$'s. The following are equivalent:
\begin{enumerate}
\item[i)] $G$ is amenable;
\item[ii)] $G$ does not contain any subgroup isomorphic to $\F_2$;
\item[ii')] $G$ does not contain any subgroup isomorphic to $\F_2$ and quasi-isometrically embedded;
\item[ii'')] $G$, when viewed as acting on the Bass-Serre tree $\tilde{X}$, does not contain any pair of hyperbolic elements with no common fixed points on the boundary $\partial\tilde{X}$;
\item[iii)] $G$ has a subgroup of index at most two isomorphic to $\Z^n$, or is an ascending HNN-extension $G=\HNN(\Z^n,\theta)$, where $\theta$ is an injective endomorphism of $\Z^n$.
\end{enumerate}
\end{Prop}

\begin{proof} Clearly $(iii)\Rightarrow (i) \Rightarrow (ii)\Rightarrow (ii')$.

$(ii')\Rightarrow (ii'')$ The contrapositive immediately follows from the ``quasi-isometric ping-pong lemma'', see Lemma 2.3 in \cite{CorTes}.

$(ii'')\Rightarrow(iii)$ Assume that $G$ does not contain any pair of hyperbolic elements with disjoint fixed points on the boundary. Then by the main result in \cite{PaVa}, the group $G$ fixes either a vertex in $\tilde{X}$, or a point in $\partial\tilde{X}$, or a pair of points in $\partial\tilde{X}$. In the first case, $G$ coincides with a vertex stabilizer, i.e. $G\simeq\Z^n$. If $G$ fixes exactly one boundary point of $\tilde{X}$, then $G$ is an ascending HNN-extension $G=\HNN(\Z^n,\theta)$, with $\theta$ not onto, by Lemma 17 in \cite{CorICC}. If $G$ fixes a pair of boundary points of $\tilde{X}$, then by Lemma 18 in \cite{CorICC} $G$ is either an amalgamated product with indices 2 and 2 and admits $\Z^n$ as subgroup of index 2, or is an HNN-extension $G=\HNN(\Z^n,\theta)$, with $\theta$ an isomorphism, which is a particular case of an ascending HNN extension.
\end{proof}

\subsection{Construction of the homomorphism $\mu$}

Our aim is to construct a non-trivial homomorphism $\mu:G\rightarrow\R^n\rtimes \textnormal{GL}_n(\R)$, mapping $H_d$ to $\textnormal{GL}_n(\R)$ and all vertex groups to $\R^n$.

We fix a maximal tree $T$ in $X$, an orientation $A$, and a base-vertex $v_0\in V$. For $v\in V$, we denote by $[v_0,v]$ the set of edges in $A\cap T$ separating $v$ from $v_0$ in $T$. For $v\in V$ and $e\in A$, we define

$$\varepsilon_{v}(e)= \left\{
\begin{array}{ll}
0 & \mbox{if $e\notin [v_0,v]$;}\\
+1 & \mbox{if $e\in [v_0,v]$ and $e$ points away from $v_0$;} \\
-1 & \mbox{if $e\in [v_0,v]$ and $e$ points towards $v_0$.}
\end{array}\right. $$

For $v\in V$, we define a matrix $\tau_v\in M_n(\Q)$ by the formula
$\tau_v:=\prod_{e\in[v_0,v]}^\sim(\sigma_{\overline{e}}\circ\sigma_e^{-1})^{\varepsilon_v(e)},$ where $\prod^\sim$ denotes the ordered product: edges on $[v_0,v]$ are ordered as they appear when running from $v_0$ to $v$ in $T$. It is then easy to see that, for every edge $e\in A\cap T$:
$\tau_{e_-}\circ\sigma_{\overline{e}}=\tau_{e_+}\circ\sigma_e$.  We then define $\mu|_{G_v}=\tau_v$, so that the relation $\sigma_e(g_e)=\sigma_{\overline{e}}(g_e)$ is satisfied for $e\in A\cap T$ and $g_e\in G_e$.

For $e\in A\smallsetminus T$, we set $\mu(t_e)=\tau_{e_-}\sigma_{\overline{e}}\sigma_e^{-1}\tau_{e_+}^{-1}\in \textnormal{GL}_n(\R)$, so that  $\mu(t_e)(\mu(\sigma_e(g_e)))=\mu(\sigma_{\overline{e}}(g_e))$ for $g_e\in G_e$, i.e.\ the relation $t_e\sigma_e(g_e)t_e^{-1}=\sigma_{\overline{e}}(g_e)$ is satisfied. So we get a homomorphism $\mu:G\rightarrow\R^n\rtimes \textnormal{GL}_n(\R)$ satisfying:

\begin{itemize}
\item $\mu(G_v)$ is a lattice in $\R^n$ for every $v\in V$;
\item $\mu(t_e)\in \textnormal{GL}_n(\R)$ for $e\in A\smallsetminus T$.
\end{itemize}

This homomorphism $\mu$ clearly factors through a homomorphism $\nu: G\rightarrow\R^n\rtimes_\mu H_d$. 

\begin{Lem}\label{Kermufree} $\Ker\,\mu$ is a free subgroup of $G$.
\end{Lem}

\begin{proof} By construction, $\Ker\;\mu$ intersects trivially every vertex group $G_v$, so it acts freely, in an orientation-preserving way, on the Bass-Serre tree $\tilde{X}$.
\end{proof}

\begin{Prop}\label{virtsol} The following are equivalent:
\begin{enumerate}
\item[i)] $\mu(H_d)$ is virtually solvable;
\item[ii)] $G$ is free-by-amenable.
\end{enumerate}
\end{Prop}

\begin{proof} The implication $(i)\Rightarrow (ii)$ is clear from lemma \ref{Kermufree}. For the converse, assume that $\mu(H_d)$ is not virtually solvable. We are going to show that, if $N\triangleleft G$ is a normal subgroup such that $G/N$ is amenable, then $N$ is {\it not} free. 

Denote by $\rho$ the composite of $\mu:H_d\rightarrow \textnormal{GL}_n(\Q)$ with the quotient map $\varphi:G\twoheadrightarrow H_d$. Let $L$ be the Zariski closure of $\mu(H_d)=\rho(G)$ in $\textnormal{GL}_n(\Q)$, 
and let $Q$ be the Zariski closure of $\rho(N)$. Then $Q$ is normal in $L$, and $L/Q$ admits a Zariski dense embedding of the amenable group $G/N$, so $L/Q$ is virtually solvable. Since, because of the assumption, $L$ is not virtually solvable, we deduce that $Q$ is not virtually solvable.

Let $S_\Q$ be the subspace of $\Q^n$ generated by the vectors $sw-w$ for $w\in\Q^n$ and $s\in Q$ (so that $\Q^n/S_\Q$ is the space of co-invariants of $Q$ in $\Q^n$). Observe that $\dim S_\Q\geq 2$ (otherwise, using a block decomposition, one sees that $Q$ is metabelian, hence virtually solvable). Let $v_0\in V$ be the base vertex used in the definition of the homomorphism $\mu$, so that $\mu(G_{v_0})=\Z^n$. Set $S=G_{v_0}\cap \mu^{-1}(S_\Q\cap\Z^n)$. We claim that $S$ is virtually contained in $N$: so $N$ contains a free abelian group of rank $\geq 2$, hence $N$ cannot be free.

To prove the claim: by Zariski density of $\rho(N)$ in $Q$, we can write every vector $w\in S_\Q$ as a finite sum $w=\sum_i \rho(n_i)w_i-w_i$, where $w_i\in\Q^n$ and $n_i\in N$. Using the fact that $\rho(G)$ commensurates $\Z^n$, this means that for every $w\in S$ there is a positive integer $k$ such that $w^k$ is a finite product of commutators: $w^k=\prod_i[n_i,w_i^k]$ with $w_i\in G_{v_0}$; so $w^k\in N$, hence $S$ is virtually contained in $N$.
\end{proof}

\subsection{Embedding $G$ as a lattice}\label{groupW}

Let $\psi:\R^n\rtimes_\mu H_d\rightarrow H_d$ be the quotient map. As in the previous section, let $\overline{\iota(G)}$ be the closure of the image of $G$ in $\textnormal{Aut}(\tilde{X})$, and $\bar{\varphi}:\overline{\iota(G)}\rightarrow H_d$ the homomorphism given by Lemma \ref{phiextends}. We define a closed subgroup $W$ of the product $\overline{\iota(G)}\times (\R^n\rtimes_\mu H_d)$ by:
$$W=\{(g,h)\in \overline{\iota(G)}\times (\R^n\rtimes_\mu H_d): \bar{\varphi}(g)=\psi(h)\};$$
i.e., $W$ is the fiber product of $\bar{\varphi}$ and $\psi$. As $\bar{\varphi}\circ\iota=\psi\circ\nu=\varphi$, we see that the natural continuous homomorphism $\iota\times\nu$ maps $G$ into $W$. The next lemma shows in particular that $G$ can be identified to a discrete subgroup in the product of the totally disconnected group $\textnormal{Aut}(\tilde{X})$ and the non-connected Lie group $\R^n\rtimes_\mu H_d$.

\begin{Prop}\label{Glattice} The homomorphism
\[\iota\times\nu\; :\; G\;\to\;\Aut(\tilde{X})\times (\R^n\rtimes_\mu H_d)\] is injective and its image is a cocompact lattice in $W$. 
\end{Prop} 

We will need the following general lemma.

\begin{Lem}\label{properness} Let $\Gamma$ be a discrete group, acting isometrically on a connected locally finite graph $X$ and a metric space $Y$. The $\Gamma$-action on $X\times Y$ is metrically proper if and only if, for some vertex $x\in X$, the stabilizer $\Gamma_x$ acts metrically properly on $Y$.
\end{Lem}

\begin{proof} The direct implication is trivial. For the converse, suppose that the $\Gamma$-action on $X\times Y$ is not metrically proper. So there exists infinitely many elements $g_n\;(n\geq 0)$ in $\Gamma$ such that, for some $(x,y)\in X\times Y$, the set $\{(g_nx,g_ny):n\geq 0\}$ is bounded. As $X$ is locally finite, passing to a subsequence we may assume that the sequence $(g_nx)_{n\geq 0}$ is constant. Set then $h_n:=g_0^{-1}g_n$, so that $h_n\in\Gamma_x$, and the sequence $(h_ny)_{n\geq 0}$ is bounded in $Y$, so that $\Gamma_x$ does not act metrically properly on $Y$.
\end{proof}

\begin{proof}[Proof of Proposition \ref{Glattice}]
 Since the kernel of $\iota$ is contained in all edge groups and $\nu$ is injective on edge groups, we see that $\iota\times\nu$ is injective. Since $\R^n\rtimes_\mu H_d$ is compactly generated, we may equip it with the word length associated with some compact generating set. This way we view $G$, via the diagonal action, as a group of isometries of $\tilde{X}\times (\R^n\rtimes_\mu H_d)$, and to prove discreteness of $G$ in $W$, it is enough to show that $G$ acts metrically properly on  $\tilde{X}\times (\R^n\rtimes_\mu H_d)$. By Lemma \ref{properness}, it is enough to show that, for some $v\in V$, the group $G_v$ acts metrically properly on  $\R^n\rtimes_\mu H_d$. But $\mu(G_v)$ is a lattice in $\R^n$, so the latter fact is obvious.

To show cocompactness of $G$ in $W$, let $v$ be a vertex in $X$, let $\tilde{v}$ be the vertex of $\tilde{X}$ whose stabilizer in $G$ is $G_v$, and let $\overline{\iota(G)}_{\tilde{v}}$ be the stabilizer of $\tilde{v}$ in $\overline{\iota(G)}$; so $\overline{\iota(G)}_{\tilde{v}}$ is a compact subgroup of $\overline{\iota(G)}$. Let $D$ be a compact subset of $\R^n$ which is a fundamental domain for $\mu(G_v)$ in $\R^n$. Consider the compact subset $K:=\overline{\iota(G)}_{\tilde{v}}\times(D\times\{1\})$ in $\textnormal{Aut}(\tilde{X})\times(\R^n\rtimes_\mu H_d)$. We claim that $W=G.K$. To see this, take $(h,y)\in W$. There exists $g_1\in G$ such that $h(\tilde{v})=g_1(\tilde{v})$, so $g_1^{-1}(h,y)\in \overline{\iota(G)}_{\tilde{v}}\times(\R^n\times\{1\})$. So we can find $g_2\in G_v$ such that $g_1^{-1}(h,y)\in \overline{\iota(G)}_{\tilde{v}}\times((\mu(g_2)+D)\times\{1\})=\overline{\iota(G)}_{\tilde{v}}\times\mu(g_2)(D\times\{1\})=g_2K$. So $W\subset G.K$. The converse inclusion is clear.
\end{proof}

\section{Proof of Theorem \ref{Mainthm}}

\subsection{The ``Haagerup'' part}

\begin{Thm}\label{Fundamental} Let $G$ be the Bass-Serre fundamental group of a finite graph of $\Z^n$'s. The following are equivalent:
\begin{enumerate}
\item[i)] $G$ has the Haagerup Property;
\item[ii)] $\R^n\rtimes_\mu H_d$ has the Haagerup Property;
\item[iii)] The closure $\overline{\mu(H_d)}$ of $\mu(H_d)$ in $\textnormal{GL}_n(\R)$, is amenable.
\end{enumerate}
\end{Thm}

\begin{proof}
 $(i)\Rightarrow(ii)$ Denote by $i$ the inclusion of $H_d$ into $\overline{\iota(G)}$, so that $\bar{\varphi}\circ i=\textnormal{Id}_{H_d}$. Observe that $i(H_d)$ is a discrete subgroup of $\overline{\iota(G)}$. The homomorphism $\R^n\rtimes_\mu H_d\rightarrow W:y\mapsto (i(\psi(y)),y)$ identifies $\R^n\rtimes_\mu H_d$ with a closed subgroup of $W$ (namely the graph of the continuous homomorphism $i\circ \psi$). So, if $G$ has the Haagerup Property, then so does $W$ (by Proposition \ref{Glattice} above, and Proposition 6.1.5 in \cite{CCJJV}). So $\R^n\rtimes_\mu H_d$ has the Haagerup Property.

$(ii)\Rightarrow(i)$ Assume $\R^n\rtimes_\mu H_d$ has the Haagerup Property. Since the tree $\tilde{X}$ is locally finite, the group $\textnormal{Aut}(\tilde{X})$ has the Haagerup Property (see e.g.\ section 1.2.3 in \cite{CCJJV}), so the product $\textnormal{Aut}(\tilde{X})\times(\R^n\rtimes_\mu\F_d)$ is Haagerup. By Proposition \ref{Glattice}, the group $G$ is a discrete subgroup in this product.

The equivalence $(ii)\Leftrightarrow(iii)$ follows from $(i)\Rightarrow(iv)$ in Proposition \ref{relT}, taking into account the fact that $H_d$ has the Haagerup Property. 
\end{proof} 

\subsection{The ``weak amenability'' part}

\begin{Thm} Let $G$ be the Bass-Serre fundamental group of a finite graph of $\Z^n$'s, as in the previous section. The following are equivalent:
\begin{enumerate}
\item[i)] $G$ is weakly amenable;
\item[ii)] $\Lambda(G)=1$; 
\item[iii)] $\overline{\mu(H_d)}$ is amenable;
\item[iv)] $\Lambda(\R^n\rtimes_\mu H_d)=1$;
\end{enumerate}
\end{Thm} 

\begin{proof} $(iii)\Rightarrow (iv)$ We embed $\R^n\rtimes_\mu H_d$ as a closed subgroup of $ H_d \times (\R^n\rtimes\overline{\mu(H_d)})$ by $x\mapsto (\psi(x),x)$.  By our assumption, the group $\R^n\rtimes\overline{\mu(H_d)}$ is amenable. Now we invoke the following results:
\begin{itemize}
\item If $M$ is a closed subgroup in $L$, then $\Lambda(M)\leq\Lambda(L)$ (Proposition 1.3 in \cite{CH});
\item $\Lambda(G_1\times G_2)=\Lambda(G_1)\Lambda(G_2)$ (Corollary 1.5 in \cite{CH});
\item $\Lambda(G)=1$ if $G$ is amenable \cite{Lep};
\item$\Lambda(H_d)=1$ (recall that $H_d$ is free of rank $d$), see \cite{Szw}.
\end{itemize}
Hence $\Lambda(\R^n\rtimes_\mu H_d)=1$.

$(iv)\Rightarrow(ii)$ By Proposition \ref{Glattice}, we may view $G$ as a discrete subgroup of $\textnormal{Aut}(\tilde{X}) \times (\R^n\rtimes\mu(H_d))$. Since $\Lambda(\textnormal{Aut}(\tilde{X}))=1$ (see \cite{Szw}), we have $\Lambda(G)=1$.

$(ii)\Rightarrow(i)$ is trivial.

$(i)\Rightarrow (iii)$ If $G$ is weakly amenable, then so is $W$, by our Proposition \ref{Glattice} and Proposition 6.2 in \cite{CH}. Since $\R^n\rtimes_\mu H_d$ can be realized as a closed subgroup of $W$ (see the proof of $(i)\Rightarrow(ii)$ in Theorem \ref{Fundamental}), we also have $\Lambda(\R^n\rtimes_\mu H_d)<+\infty$. By $(v)\Rightarrow (i)$ in Proposition \ref{relT}, the closure $\overline{\mu(H_d)}$ is amenable.
\end{proof}

\section{Quasi-isometric embeddings in $L^p$-spaces}

\subsection{Preliminaries}\label{prel}

We keep notations as in subsection \ref{groupW}.  Observe that the group $W$ introduced there is compactly generated, since by Proposition \ref{Glattice} it contains the finitely generated group $G$ as a co-compact lattice.

\begin{Lem}\label{biLip1} The inclusion of $W$ into $\overline{\iota(G)}\times (\R^n\rtimes_\mu H_d)$ is a quasi-isometric embedding.
\end{Lem}

\begin{proof} We have two large-scale Lipschitz embeddings:
$$\R^n\rtimes_\mu H_d\rightarrow W:y\mapsto (i(\psi(y)),y)$$
(already considered in the proof of Theorem \ref{Fundamental}) and 
$$\overline{\iota(G)}\rightarrow W: g\mapsto (g,(0,\bar{\varphi}(g))).$$
So, to estimate the distance in $W$ between the points $(1,(0,1))$ and $(g,(x,w))\in W$, we insert the intermediate point $(i(w),(x,w))\in W$: the points $(1,(0,1))$ and $(i(w),(x,w))$ are in the image of the first embedding, while the points $(1,(0,1))$ and $(i(w)^{-1}g,(0,1))$ are in the image of the second embedding; by the triangle inequality and left invariance:
\begin{align*}
 & d_W((1,(0,1));(g,(x,w)))\\
 & \leq  d_W((1,(0,1));(i(w),(x,w)))+d_W((i(w),(x,w));(g,(x,w)))\\
& = d_W((1,(0,1));(i(w),(x,w)))+d_W((1,(0,1)); (i(w)^{-1}g,(0,1)))\\
&\leq  C_1d_{\R^n\rtimes_\mu H_d}((0,1),(x,w))+C_2d_{\overline{\iota(G)}}(1,i(w)^{-1}g)\\
& =  C_1d_{\R^n\rtimes_\mu H_d}((0,1),(x,w))+C_2d_{\overline{\iota(G)}}(i(\bar{\varphi}(g)),g)\\
& \leq  C_1d_{\R^n\rtimes_\mu H_d}((0,1),(x,w))+C'_2d_{\overline{\iota(G)}}(1,g),\end{align*}
hence the result.
\end{proof}

We reach some cases where Conjecture \ref{quest} holds for $\R^n\rtimes_\mu H_d$.

\begin{Prop}\label{specialcases} In each of the cases $d=0$, $d=1$, and $n=1$, the inclusion of $\R^n\rtimes_\mu H_d$ into $\R^n\rtimes\overline{\mu(H_d)}$ is a quasi-isometry in restriction to $\R^n$.
\end{Prop}

\begin{proof} We check that the assumption of (a) in Proposition \ref{biLip2} is satisfied, namely that there exists a finitely generated subgroup $\Lambda\subset H_d$ such that $\mu(\Lambda)$ is a cocompact lattice in $\overline{\mu(H_d)}$. The case $d=0$ is trivial. The case $d=1$ follows from the well-known fact that an infinite cyclic subgroup of $\textnormal{GL}_n(\R)$ is either closed (take $\Lambda=H_d$) or relatively compact (take $\Lambda$ trivial). Finally, for $n=1$, take $\Lambda$ to be trivial if $\mu$ has finite image, or $\Lambda$ to be infinite cyclic if $\mu$ has infinite image. 
\end{proof}

\subsection{Compression estimates}

We have the following estimates on $L^p$-compression. Recall that the few cases where $G$ is amenable, are completely described in Proposition \ref{amen}.

\begin{Thm}\label{Lcomp} \begin{enumerate}
\item[a)] Assume that there exists a closed, almost connected subgroup $L\subset \textnormal{GL}_n(\R)$, that contains $\overline{\mu(H_d)}$ as a co-compact subgroup; and that Conjecture \ref{quest} has a positive answer in this case. Then $\alpha_p(G)=1$.
\item[b)] Assume moreover that $\overline{\mu(H_d)}$ is amenable. 
\begin{itemize}
\item If $G$ is amenable: $\alpha^\sharp_p(G)=1$.
\item If $G$ is non-amenable: $\alpha^\sharp_p(G)=\max\{1/p,1/2\}.$
\end{itemize}
\end{enumerate}
\end{Thm}

\begin{proof} \begin{itemize}
\item[a)] By Proposition \ref{Glattice} and lemma \ref{biLip1}, $G$ embeds quasi-isometrically in $W$ and $W$ embeds quasi-isometrically in $Aut(\tilde{X})\times (\R^n\rtimes_\mu H_d)$. By both of our assumptions, the homomorphism $\R^n\rtimes_\mu H_d \rightarrow \F_d\rtimes (\R^n\rtimes L):x\mapsto (\varphi(x),\mu(x))$ is a quasi-isometric embedding. So we have:
$$\alpha_p(G)\geq \alpha_p(\textnormal{Aut}(\tilde{X}) \times\F_d\times (\R^n\rtimes L))=\min\{\alpha_p(\tilde{X}),\alpha_p(\F_d),\alpha_p(\R^n\rtimes L)\}.$$
 But $\alpha_p(\tilde{X})=\alpha_p(\F_d)=1$; this is, for instance, a consequence of \cite[Corollary 2]{Tessera}; for $p=2$ this was done by Guentner and Kaminker \cite{GK}. On the other hand $\alpha_p$ of an almost connected Lie group is 1, by \cite[Theorem 1]{Tessera}.
 \item[b)] If $G$ is amenable, then $\alpha_p^\sharp(G)=\alpha_p(G)=1$, where the first equality follows from amenability of $G$ and Theorem 9.1 in \cite{NP2}.
 
 If $G$ is non-amenable, by Theorem 1.1 in \cite{NP1} we have $\alpha^\sharp_p(G)\leq\max\{1/p,1/2\}.$ To prove the opposite inequality, observe that, since the quasi-isometric embedding $G\hookrightarrow \textnormal{Aut}(\tilde{X}) \times\F_d\times (\R^n\rtimes L))$ is $G$-equivariant, we have

\begin{align*}\alpha_p^\sharp(G)\geq  & \alpha_p^G(\textnormal{Aut}(\tilde{X}) \times\F_d\times (\R^n\rtimes L))\\
 & \geq\min\{\alpha_p^\sharp(\textnormal{Aut}(\tilde{X})),\alpha_p^\sharp(\F_d),\alpha_p^\sharp(\R^n\rtimes L)\}.\end{align*}
 Since $\tilde{X}$ satisfies a strong isoperimetric inequality, we have 
 $\alpha^\sharp_p(\textnormal{Aut}(\tilde{X}))=\max\{1/p,1/2\}$, by Lemma 2.3 in \cite{NP1}; moreover, as $\R^n\rtimes L$ is amenable:  $\alpha_p^\sharp(\R^n\rtimes L)=\alpha_p(\R^n\rtimes L)=1$, where the first equality follows from Theorem 9.1 in \cite{NP2}.\qedhere
 \end{itemize}
\end{proof}

{\bf Example \ref{goodexamples}(2), continued:} Because of the relative Property (T) for the pair $(\Z^2\rtimes\F_2,\Z^2)$, we have $\alpha_2^\sharp(\Z^2\rtimes\F_2)=0$. In this case Conjecture \ref{quest} holds, by Proposition \ref{biLip2}. This shows that at least part (b) of Theorem \ref{Lcomp} does not hold without further assumptions.

\begin{Cor} The conclusions of Theorem \ref{Lcomp} hold if either $d=0$, or $d=1$, or $n=1$.
\end{Cor}

\begin{proof} Conjecture \ref{quest} holds in these cases, by Proposition \ref{specialcases}. The existence of a closed, amenable, almost connected subgroup $L$ containing $\overline{\mu(H_d)}$ as a co-compact subgroup, is obvious for $d=0$; for $d=1$ it follows from the fact that, for any matrix $A\in \textnormal{GL}_n(\R)$, some power of $A$ belongs to a 1-parameter subgroup; and for $n=1$ it follows from the structure of closed subgroups of $\textnormal{GL}_1(\R)$.
\end{proof}

\section{Quasi-isometric embeddings into products of trees}

The question here is whether our group $G$ embeds quasi-isometrically into a finite product of 3-regular trees. We deal with the same cases as in Corollary \ref{suffHaagerup}. 

\subsection{The case $d=0$}

\begin{Prop}\label{d=0} If $d=0$, then $G$ embeds quasi-isometrically into the product of $n+1$ copies of the 3-regular tree.
\end{Prop}

\begin{proof} By Lemma \ref{biLip1}, the group $G$ embeds quasi-isometrically into $\tilde{X}\times\R^n$.
\end{proof}

\subsection{The case $d=1$}

Example \ref{Heis} shows that, in order to embed our group $G$ quasi-isometrically into a finite product of trees, we need some additional assumption on eigenvalues of $\mu$.

\begin{Prop}\label{eigenvalues} Assume that $\mu$ has no eigenvalue with modulus 1. Then $G$ embeds quasi-isometrically into the product of $n+3$ copies of the 3-regular tree.
\end{Prop}

\begin{proof} By Lemma \ref{biLip1}, it is enough to prove that $H=:\R^n\rtimes_\mu\Z$ embeds quasi-isometrically into the product of $n+2$ copies of the 3-regular tree. First observe that replacing $H$ by a finite index subgroup and replacing $\mu$ by a proper power, we can suppose that $\mu=\mu_1$ lies in a one-parameter subgroup $(\mu_t)$. 

We will embed $H$ quasi-isometrically into a product of two simply connected homogeneous negatively curved Riemannian manifolds $S_+$, $S_-$, of dimensions $n_+$, $n_-$ with $n_++n_-=n+2$. We are then in position to apply \cite[Theorem 1.2]{BDS}: the visual boundary $\partial_\infty(S_\pm)$ has dimension $n_\pm-1$ and is a doubling metric space (this follows e.g.\ from \cite[Theorem 9.2]{BoSc} and the discussion following it), so $S_\pm$ embeds quasi-isometrically into a product of $n_\pm$ copies of the 3-regular tree. 

Let us first assume that all eigenvalues of $\mu$ have simultaneously modulus $>1$ or $<1$. Let $S$ be the semidirect product $S=\R^n\rtimes_{\mu_t}\R$. By a result of Heintze (see \cite[Theorem 3]{Hein}), the simply connected solvable group $S$ carries a left-invariant Riemannian metric with negative curvature.

Turning back to the general hypotheses, consider the decomposition of $\R^n$ into characteristic subspaces of $\mu$. We can gather them to get a decomposition
$\R^n=V_+\oplus V_-$, where $V_+$ (resp.\ $V_-$) consists of the sum of characteristic subspaces with eigenvalues of modulus $>1$ (resp.\ $<1$). This provides a quasi-isometric homomorphic 
embedding of $H\subset H_+\times H_-$, where $H_\pm=V_\pm\rtimes\Z$. Set $n_\pm=1+\dim V_\pm$. By the previous case, $H_\pm$ embeds quasi-isometrically into a product $S_+\times S_-$, where $\dim S_\pm=n_\pm$ and $S_\pm$ is a simply connected homogeneous negatively curved Riemannian manifold.
\end{proof}

\subsection{The case $n=1$}

When $n=1$, the homomorphism $\mu$ takes image in the multiplicative group $\R^\times$, so we separate two cases according to whether the image of $\mu$ is finite or infinite.

\begin{Prop}\label{n=1} Assume $n=1$. If $|\mu(w)|=1$ for every $w\in \F_d$, then $G$ embeds quasi-isometrically into a product of three 3-regular trees.
\end{Prop}

\begin{proof} $\R\rtimes_\mu\F_d$ is quasi-isometric to its finite index subgroup $\R\times\Ker\mu$, so the result follows from Lemma \ref{biLip1}.
\end{proof}

Now assume that $|\mu(w)|\neq 1$ for some $w\in H_d$, i.e.\ $\mu|_{H_d}$ has infinite image.

The ``$ax+b$"-group $\R\rtimes \R^\times$ acts transitively isometrically on the hyperbolic plane $\mathbf{H}^2=\{z\in\C: \Im(z)>0\}$ (with the action being simply transitive when restricted to the connected component of identity). Composing the homomorphism $\R\rtimes_\mu\F_d\rightarrow \R\rtimes\R^\times$ with the orbital map $\R\rtimes\R^\times\rightarrow\mathbf{H}^2$ associated with $z_0=i$, we get a map $h:\R\rtimes_\mu\F_d\rightarrow\mathbf{H}^2$. 

\begin{Lem}\label{hyperb} Assume $n=1$ and $\mu$ has infinite image. The map $(\psi,h):\R\rtimes_\mu\F_d\rightarrow \F_d\times \mathbf{H}^2$ is a quasi-isometric embedding.
\end{Lem}

\begin{proof} Endow $\R$ with the metric induced from word length on $\R\rtimes_\mu\F_d$. Clearly, it is enough to see that the map $\alpha|_{\R}:\R\rightarrow\mathbf{H}^2:t\mapsto i+t$ is a quasi-isometric embedding. Now, since $\mu$ has infinite image, the subgroup $\R$ is exponentially distorted into $\R\rtimes_\mu\F_d$, i.e the word length of $t$ is equivalent to $\log|t|$ for $|t|\gg 0$. On the other hand, a well-known formula in hyperbolic geometry (see e.g.\ Theorem 7.2.1 in \cite{Beard}) gives $\sinh\left(\frac{1}{2}d_{\mathbf{H}^2}(i,i+t)\right)=\frac{|t|}{2}$, so the result follows. 
\end{proof}

\begin{Prop}\label{n=1bis} If $n=1$ and $\mu$ has infinite image, then $G$ embeds quasi-isometrically into a product of four 3-regular trees.
\end{Prop}

\begin{proof} By Lemma \ref{biLip1} and Lemma \ref{hyperb}, the group $G$ embeds quasi-isometri\-cally into $\tilde{X}\times\mathbf{H}^2\times \F_d$, where $\tilde{X}$ is the Bass-Serre tree of $G$. By \cite{BDS}, the hyperbolic plane $\mathbf{H}^2$ embeds quasi-isometrically into a product of two 3-regular trees.
\end{proof}

Theorem \ref{classofT} in the introduction follows by combining Propositions \ref{d=0}, \ref{eigenvalues}, \ref{n=1} and \ref{n=1bis}. When $G$ is non-amenable, $G\apx\mathcal{T}$ follows from Proposition \ref{amen}.

\end{document}